\documentclass[12pt]{amsart}

\usepackage[margin=2cm]{geometry}
\usepackage[english]{babel}
\usepackage{amsmath}
\usepackage{amsthm}
\usepackage{amssymb}
\usepackage{enumerate}
\usepackage{ucs}
\usepackage[utf8]{inputenc}
\usepackage[T1]{fontenc}
\usepackage{color}
\usepackage{ stmaryrd }
\usepackage{hyperref}
\usepackage[retainorgcmds]{IEEEtrantools}
\usepackage{tikz}
\usepackage{tikz-cd}
\usetikzlibrary{arrows}
\usepackage[initials, shortalphabetic]{amsrefs}

\pgfdeclarelayer{invisible}
\pgfdeclarelayer{edgelayer}
\pgfdeclarelayer{nodelayer}
\pgfsetlayers{edgelayer,nodelayer,invisible,main}

\tikzstyle{empty}=[inner sep=0pt,scale=0.05,draw=none]
\tikzstyle{A}=[thick,fill=white,circle,draw,minimum height=0.8cm,inner sep=0pt]
\tikzstyle{cbi}=[circle,scale=0.5, fill=black]
\tikzstyle{bary}=[]
\tikzstyle{patate}=[thick]

  \newcommand{\N}{\mathbb N}
  
  \newcommand{\K}{\mathbf K}
  \newcommand{\M}{\mathbf M}
  \newcommand{\Lang}{\mathcal L}
  \newcommand{\Kl}{\mathcal K}
  \newcommand{\A}{\mathbf A}
  \newcommand{\B}{\mathbf B}
  \newcommand{\C}{\mathbf C}
  \newcommand{\I}{[0,1]}
  \newcommand{\Ul}{\mathcal U}
  \newcommand{\U}{\mathbb U}

 \newcommand{\bit}{\begin{itemize}}
 \newcommand{\eit}{\end{itemize}}
  \newcommand{\ben}{\begin{enumerate}}
 \newcommand{\een}{\end{enumerate}}
 \newcommand{\eps}{\epsilon}

 \newcommand{\emb}[2]{{}^{\mathbf #1}{\mathbf #2}}

 \theoremstyle{definition}
\newtheorem{thm}{Theorem}
\newtheorem{cor}[thm]{Corollary}

\newtheorem{prop}[thm]{Proposition}
\newtheorem{df}[thm]{Definition}

\newtheorem{rmq}[thm]{Remark}
\newtheorem*{rmq*}{Remark}
\newtheorem*{intuition}{Intuition}

  \title{Amenability and Ramsey theory in the metric setting}
\author{Adriane Kaïchouh}

\begin{document}
\reversemarginpar

\begin{abstract}
In \cite{moore-amenability-ramsey}, Moore characterizes the amenability of automorphism groups of countable ultrahomogeneous structures by a Ramsey-type property. We extend this result to automorphism groups of metric Fraïssé structures, which encompass all Polish groups. As an application, we prove that amenability is a $G_\delta$ condition. 
\end{abstract}

\maketitle

\section*{Introduction}
   
In recent years, there has been a flurry of activity relating notions linked to amenability or groups on one side, and combinatorial conditions linked to Ramsey theory on the other side. In this paper, we extend a result of Moore (\cite[theorem 7.1]{moore-amenability-ramsey}) on the amenability of closed subgroups of $S_{\infty}$ to general Polish groups. A topological group is said to be \textit{amenable} if every continuous action of the group on a compact space admits an invariant probability measure.

Moore's result is the counterpart of a theorem of Kechris, Pestov and Todor\v{c}evi\'{c} (\cite{MR2140630}) on extreme amenability. A topological group is said to be \textit{extremely amenable} if every continuous action of the group on a compact space admits a fixed point. In the context of closed subgroups of $S_\infty$, which are exactly the automorphism groups of Fraïssé structures, Kechris, Pestov and Todor\v{c}evi\'{c} characterize extreme amenability by a combinatorial property of the associated Fraïssé class (in the case where its objects are rigid), namely, the Ramsey property. A class $\Kl$ of structures is said to have \textit{the Ramsey property} if for all structures $A$ and $B$ in $\Kl$, for all integers $k$, there is a structure $C$ in $\Kl$ such that for every coloring of the set of copies of $A$ in $C$ with $k$ colors, there exists a copy of $B$ in $C$ within which all copies of $A$ have the same color. 

Thus, extreme amenability, which provides fixed points, corresponds to colorings having a "fixed", meaning monochromatic, set. Amenability, on the other side, provides invariant measures. Since a measure is not far from being a barycenter of point masses, the natural mirror image of the Ramsey property in that setting should be for a coloring to have a "monochromatic convex combination of sets". Indeed, Tsankov (in an unpublished note) and Moore introduced a \textit{convex Ramsey property} and proved that a Fraïssé class has the convex Ramsey property if and only if the automorphism group of its Fraïssé limit is amenable. 

Besides, Kechris, Pestov and Todor\v{c}evi\'{c}'s result was extended to general Polish groups by Melleray and Tsankov in \cite{julientodor-amenablecontinuouslogic}. They use the framework of continuous logic (see \cite{MR2436146}) via the observation that every Polish group is the automorphism group of an approximately homogeneous metric structure (\cite{MR2767973}), that is of a metric Fraïssé limit (in the sense of \cite{itai-metricFraisse}). They define an \textit{approximate Ramsey property} for classes of metric structures and then show that a metric Fraïssé class has the approximate Ramsey property if and only if the automorphism group of its Fraïssé limit is extremely amenable.

In this paper, we "close the diagram" by giving a metric version of Moore's result. We replace the classical notion of a coloring with the metric one (from \cite{julientodor-amenablecontinuouslogic}) to define a \textit{metric convex Ramsey property}, and we prove the exact analogue of Moore's theorem:

\begin{thm}
Let $\Kl$ be a metric Fraïssé class, $\K$ its Fraïssé limit and $G$ the automorphism group of $\K$. 
Then $G$ is amenable if and only if $\Kl$ satisfies the metric convex Ramsey property.
\end{thm}

Note that this characterization makes mention of both the group and the Fraïssé class. In the course of the proof, we provide several reformulations of the metric convex Ramsey property, among which a property that only involves the group itself. Thus, as a corollary, we obtain the following intrinsic characterization of the amenability of a Polish group.

\begin{thm}
Let $G$ be a Polish group and $d$ a left-invariant metric on $G$ which induces the topology. Then the following are equivalent.
\ben
\item The topological group $G$ is amenable.
\item For every $\eps > 0$, every finite subset $F$ of $G$, every $1$-Lipschitz map $f : (G,d) \to \I$, there exist elements $g_1,..., g_n$ of $G$ and barycentric coefficients $\lambda_1,..., \lambda_n$ such that for all $h,h' \in F$, one has
   $$\left\lvert \sum_{i=1}^n \lambda_i f(g_i h) - \sum_{i=1}^n \lambda_i f(g_i h') \right\rvert < \eps.$$
\een
\end{thm}

From these theorems, we deduce some nice structural consequences about amenability. The first one follows directly from the previous theorem and the Riesz representation theorem (applied to the Samuel compactification of $G$).

\begin{cor}
Let $G$ be a Polish group. Then the following are equivalent.
\ben
\item $G$ is amenable.
\item For every right uniformly continuous function $f : G \to \I$, there exists a positive linear form $\Lambda$ of norm $1$ on $C(G, \I)$ such that for all $g \in G$, one has $\Lambda(g \cdot f) = \Lambda (f)$.
\een
\end{cor}

It constitutes an inversion of quantifiers in the definition of amenability, meaning that to check that a Polish group is amenable, it suffices to verify it for one function at a time. Note that Moore obtained a similar result for discrete groups. Besides, the same is true for extreme amenability with multiplicative linear forms.

Another advantage of these theorems is to express amenability in a finitary way, which allows us to compute its Borel complexity. In \cite{julien-todor-genericrepresentations}, Melleray and Tsankov use an oscillation stability property (that resembles the Ramsey property, see \cite{MR2277969}) to show that extreme amenability is a $G_\delta$ condition; we prove that the same holds for amenability. From this, a Baire category argument leads to the following sufficient condition for a Polish group to be amenable.

\begin{cor}
Let $G$ be a Polish group such that for every $n \in \N^*$, the set
$$F_n = \{ (g_1,..., g_n) \in G^n  :  \langle g_1,...,g_n \rangle \text{ is amenable (as a subgroup of $G$)} \}$$
is dense in $G^n$. Then $G$ is amenable.
\end{cor}

This is a slight strengthening of the fact that a Polish group whose finitely generated subgroups are amenable is itself amenable (see \cite[theorem 1.2.7]{greenleaf}).

\section{The metric convex Ramsey property}

We use the notations and conventions of \cite{julientodor-amenablecontinuouslogic}. In particular, we assume symbols of continuous languages to be Lipschitz.

\begin{df}
Let $\Lang$ be a relational continuous language, $\A$ and $\B$ two finite $\Lang$-structures and $\M$ an arbitrary $\Lang$-structure.
\bit 
\item We denote by $\emb{A}{M}$ the set of all embeddings of $\A$ into $\M$. 
We endow $\emb{A}{M}$ with the metric $\rho_\A$ defined by 
$$\rho_\A (\alpha, \alpha') = \sup_{a \in A} d(\alpha(a), \alpha'(a)).$$

\item A \textbf{coloring} of $\emb{A}{M}$ is a $1$-Lipschitz map from $(\emb{A}{M}, \rho_\A)$ to the interval $\I$.

\item We denote by $\left\langle \emb{A}{M} \right\rangle$ the set of all finitely supported probability measures on $\emb{A}{M}$. We will identify embeddings with their associated Dirac measures. 

\item If $\kappa : \emb{A}{M} \to \I$ is a coloring, we extend $\kappa$ to $\left\langle \emb{A}{M} \right\rangle$ linearly: if $\nu$ in $\left\langle \emb{A}{M} \right\rangle$ is of the form $\displaystyle{\nu = \sum_{i=1}^n \lambda_i \delta_{\alpha_i}}$, we set
$$\kappa(\nu) = \sum_{i=1}^n \lambda_i \kappa(\alpha_i).$$

\item Moreover, we extend composition of embeddings to finitely supported measures bilinearly. Namely, if $\nu$ in $\left\langle \emb{A}{B} \right\rangle$ and $\nu'$ in $\left\langle \emb{B}{M} \right\rangle$ are of the form $\displaystyle{\nu = \sum_{i=1}^n \lambda_i \delta_{\alpha_i}}$ and $\displaystyle{\nu' = \sum_{j=1}^m \lambda'_j \delta_{\alpha'_j}}$, we define
$$\nu' \circ \nu = \sum_{j=1}^m \sum_{i=1}^n \lambda'_j \lambda_i \delta_{\alpha'_j \circ \alpha_i}.$$

\item If $\nu$ is a measure in $\langle \emb{B}{M} \rangle$, we denote by $\left\langle \emb{A}{M}(\nu) \right\rangle$ the set of all finitely supported measures which factor through $\nu$ and by $\emb{A}{M}(\nu)$ the set of those which factor through $\nu$ via an embedding. 
More precisely, if $\nu \in \left\langle \emb{B}{M} \right\rangle$ is of the form $\displaystyle{\sum_{i = 1}^n \lambda_i \delta_{\beta_i}}$, we define
$$\emb{A}{M}(\nu) = \left\{ \nu \circ \delta_\alpha : \alpha \in \emb{A}{B} \right\}$$
and 
$$\left\langle \emb{A}{M}(\nu) \right\rangle = \left\{ \nu \circ \nu' : \nu' \in \left\langle \emb{A}{B} \right\rangle \right\}.$$
\eit
\end{df}

Throughout the paper, $\Kl$ will be a metric Fraïssé class in a relational continuous language and $\K$ will be its Fraïssé limit.

   \begin{df}
   The class $\Kl$ is said to have \textbf{the metric convex Ramsey property} if for every $\eps > 0$, for all structures $\A$ and $\B$ in $\Kl$, there exists a structure $\C$ in $\Kl$ such that for every coloring $\kappa : \emb{A}{C} \to \I $, there is $\nu$ in $\langle \emb{B}{C} \rangle$ such that for all $\alpha, \alpha' \in \emb{A}{B}(\nu)$, one has $\lvert \kappa(\alpha) - \kappa(\alpha') \rvert < \eps$.
   \end{df}
   
   \begin{intuition}
   In the classical setting, the Ramsey property states that given two structures $A$ and $B$, we can find a bigger structure $C$ such that whenever we color the copies of $A$ in $C$, we can find a copy of $B$ in $C$ wherein every copy of $A$ has the same color. Here, it basically says that we can find a convex combination of copies of $B$ in $C$ wherein every compatible convex combination of copies of $A$ has almost the same color (see figure \ref{dessin}).
   \end{intuition}
   
   \begin{figure}\label{dessin}
\begin{tikzpicture}[scale=0.5]
\begin{pgfonlayer}{invisible}
		\node [style=empty] (9) at (-6, 3.5) {};
		\node [style=empty] (10) at (-6, -3.5) {};
		\node [style=empty] (11) at (6, 3.5) {};
		\node [style=empty] (12) at (6, -3.5) {};
		\node [style=empty] (13) at (11, 0) {};
		\node [style=empty] (14) at (-11, 0) {};
	\end{pgfonlayer}
	\begin{pgfonlayer}{nodelayer}
		\node [style=A] (0) at (-6, 2) {$A_1$};
		\node [style=A] (1) at (-5, 0) {$A'_1$};
		\node [style=A] (2) at (-7, -2) {$A''_1$};
		\node [style=A] (3) at (6, 2) {$A_2$};
		\node [style=A] (4) at (7, 0) {$A'_2$};
		\node [style=A] (5) at (5, -2) {$A''_2$};
		\node [style=cbi] (6) at (-2, 2) {};
		\node [style=cbi] (7) at (-1, 0) {};
		\node [style=cbi] (8) at (-2, -2) {};
		\node [above of=6,xshift=0.6cm,yshift=-0.7cm,scale=0.9] {$2/3A_1+1/3A_2$};
		\node [above of=7,xshift=0.6cm,yshift=-0.7cm,scale=0.9] {$2/3A'_1+1/3A'_2$};
		\node [above of=8,xshift=0.6cm,yshift=-0.7cm,scale=0.9] {$2/3A''_1+1/3A''_2$};
		\node  (15) at (-6, -4.25) {$(B_1,2/3)$};
		\node  (16) at (6, -4.25) {$(B_2, 1/3)$};
		\node  (17) at (0, -7.75) {$C$};

	\end{pgfonlayer}
	\begin{pgfonlayer}{edgelayer}
		\draw [style=bary] (0) to (6);
		\draw [style=bary] (6) to (3);
		\draw [style=bary] (1) to (7);
		\draw [style=bary] (7) to (4);
		\draw [style=bary] (2) to (8);
		\draw [style=bary] (8) to (5);
		\draw [style=patate, bend right=90, looseness=1.25] (9) to (10);
		\draw [style=patate, bend right=270, looseness=1.25] (9) to (10);
		\draw [style=patate, bend right=90, looseness=1.25] (11) to (12);
		\draw [style=patate, bend right=270, looseness=1.25] (11) to (12);
		\draw [style=patate, bend left=90] (14) to (13);
		\draw [style=patate, bend right=90] (14) to (13);
	\end{pgfonlayer}
\end{tikzpicture}
\caption{Black points are barycenters of two corresponding copies of $A$ in $B_1$ and $B_2$ with coefficients $2/3$ and $1/3$. The metric convex Ramsey property says that all these points almost have the same color.}
   \end{figure}
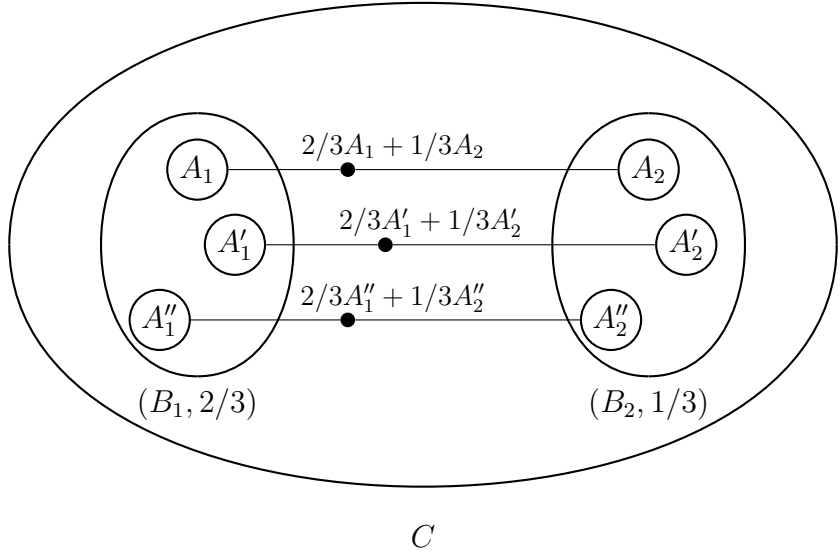

   \begin{rmq}\label{mooresecompliquelavie}
   One can replace the assumption $\alpha, \alpha' \in \emb{A}{B}(\nu)$ with the stronger one $\alpha, \alpha' \in \langle \emb{A}{B}(\nu) \rangle$ in the above definition, as is done in \cite{moore-amenability-ramsey}. Indeed, the property is preserved by barycenter. 
    \end{rmq}
   
   The following proposition states that the metric convex Ramsey property allows us to stabilize any finite number of colorings at once.
   
   \begin{prop}
   The following are equivalent.
   \ben
   \item The class $\Kl$ has the metric convex Ramsey property.
   \item For every $\eps > 0$, for all integers $N \in \N^*$ and all structures $\A$ and $\B$ in $\Kl$, there exists a structure $\C$ in $\Kl$ such that for all colorings $\kappa_1, ..., \kappa_N : \emb{A}{C} \to \I$, there is $\mu$ in $\langle \emb{B}{C} \rangle$ such that for all $j$ in $\{1,..., N\}$ and all $\alpha, \alpha'$ in $\emb{A}{C}(\mu)$, one has $\lvert \kappa_j(\alpha) - \kappa_j(\alpha') \rvert < \eps$.
   \een
   \end{prop}
   
   \begin{rmq*}
   Condition $(2)$ above is equivalent to the metric convex Ramsey property for colorings into $\I^N$, where $\I^N$ is endowed with the supremum metric. It follows that the metric convex Ramsey property is equivalent to the same property for colorings taking values in any convex compact metric space.
   \end{rmq*}
   
   \begin{proof}
  The second condition clearly implies the first. For the other implication, let $N$ be a positive integer, $\A$ and $\B$ structures in $\Kl$ and $\eps > 0$. Put $\C_{-1} = \A$ and $\C_0 = \B$. By induction, we find structures $\C_1,...,\C_N$ in $\Kl$ witnessing the metric convex Ramsey property for $\C_{i-1}$, $\C_i$ and $\eps$, that is, such that for every $j \in \{1,..., N\}$, if $\kappa : \emb{A}{C_j} \to \I$ is a coloring, then there exists $\nu \in \left\langle \emb{C_{j-1}}{C_j} \right\rangle$ such that for all $\alpha, \alpha'$ in $\emb{A}{C_j}(\nu)$, we have $\lvert \kappa(\alpha) - \kappa(\alpha') \rvert < \eps$.
  
  Set $\C = \C_N$. We show that $\C$ has the desired property. To this aim, let $\kappa_1,..., \kappa_N : \emb{A}{C} \to \I$ be colorings. 
  
  By downward induction, we build $\mu_N,...,\mu_0$ such that 
  \bit
  \item $\mu_n = \delta_{\text{id}_\C}$ (we can identify $\mu_n$ with $\C$),
  \item for every $j < N$, $\mu_i \in \left\langle \emb{C_j}{C}(\mu_{j+1}) \right\rangle$ (in particular, $\mu_{N-1} \in \left\langle \emb{C_{N-1}}{C} \right\rangle$ under the previous identification), 
  \item for every $j < N$, if $\alpha, \alpha' \in \emb{A}{C}(\mu_j)$, then $\lvert \kappa_j(\alpha) - \kappa_j(\alpha') \rvert < \eps$.
  \eit 
  Suppose $\mu_{j+1}$ has been constructed. We lift the coloring $\kappa_j$ to $\tilde{\kappa_j} : \emb{A}{C_{j+1}} \to \I$ by putting $\tilde{\kappa_j}(\alpha) = \kappa_j(\mu_{j+1} \circ \delta_\alpha)$. The map $\tilde{\kappa_j}$ we obtain is again a coloring. 
  
  
  Therefore, we may apply our assumption on $\C_{j+1}$ to $\tilde\kappa_j$: there exists $\nu \in \left\langle \emb{C_j}{C_{j+1}} \right\rangle$ such that for all $\alpha,\alpha' \in \emb{A}{C_{j+1}}(\nu)$, $\lvert \tilde \kappa_j(\alpha) - \tilde \kappa_j(\alpha') \rvert < \eps$. Then $\mu_j = \mu_{j+1} \circ \nu$ is as desired. Indeed, let $\alpha, \alpha' \in \emb{A}{C}(\mu_j)$. There exist $\tilde \alpha, \tilde \alpha' \in \emb{A}{C_j}$ such that $\alpha = \mu_j \circ \delta_{\tilde \alpha}$ and $\alpha' = \mu_j \circ \delta_{\tilde \alpha'}$. Then 
  \begin{align*}
  \lvert \kappa_j(\alpha) - \kappa_j(\alpha') \rvert
  &= \lvert \kappa_j(\mu_j \circ \delta_{\tilde \alpha}) - \kappa_j(\mu_j \circ \delta_{\tilde \alpha'}) \rvert \\
  &= \lvert \kappa_j(\mu_{j+1} \circ \nu \circ \delta_{\tilde \alpha}) - \kappa_j(\mu_{j+1} \circ \nu \circ \delta_{\tilde \alpha'}) \rvert \\
  &= \lvert \tilde \kappa_j(\nu \circ \delta_{\tilde \alpha}) - \tilde \kappa_j(\nu \circ \delta_{\tilde \alpha'}) \rvert \\
  &< \eps,
  \end{align*}
  thus completing the construction.
  
  Finally, put $\mu = \mu_0 \in \left\langle \emb{B}{C} \right\rangle$. Then $\mu$ is as desired. Indeed, the construction of the $\mu_j$'s gives that for every $j \in \{1,..., N\}$,
  $$\mu \in \left\langle \emb{B}{C}(\mu_1) \right\rangle \subseteq \left\langle \emb{B}{C}(\mu_2) \right\rangle \subseteq ... \subseteq \left\langle \emb{B}{C}(\mu_j) \right\rangle$$
  so that whenever $\alpha, \alpha' \in \emb{A}{C}(\mu)$, they are in $\emb{A}{C}(\mu_j)$ too and the assumption on $\mu_j$ yields that $\vert \kappa_j(\alpha) - \kappa_j(\alpha') \rvert < \eps$.
   \end{proof}
   
   We now give a infinitary reformulation of the metric convex Ramsey property which is the convex counterpart of the classical infinite Ramsey property.
      
   \begin{rmq}\label{Ncolorings}
    For the sake of simplicity, we state the results for only one coloring at a time; the previous proposition will imply that we can do the same with any finite number of colorings.
   \end{rmq}
   
   \begin{prop}\label{ramseyK}
   The following are equivalent.
   \ben
   \item The class $\Kl$ has the metric convex Ramsey property.
   \item For every $\eps > 0$, for all structures $\A$ and $\B$ in $\Kl$ and all colorings $\kappa : \emb{A}{K} \to \I$, there exists $\nu$ in $\langle \emb{B}{K} \rangle$ such that for all $\alpha, \alpha'$ in $\emb{A}{K}(\nu)$, one has $\lvert \kappa(\alpha) - \kappa(\alpha') \rvert < \eps$.
   \een
   \end{prop}
   
   When we will prove, in theorem \ref{main}, that the metric convex Ramsey property is implied by amenability, it will be in the guise of condition $(2)$ above.
   
   \begin{proof}
     $(1) \Rightarrow (2)$] Fix $\eps > 0$, $\A$ and $\B$ two structures in $\Kl$ and let $\C \in \Kl$ witness the metric convex Ramsey property for $A$, $B$ and $\eps$. We may assume that $\C$ is a substructure of $\K$. Now every coloring of $\emb{A}{K}$ restricts to a coloring of $\emb{A}{C}$ so, if $\nu$ is the measure given by $\C$ for a coloring $\kappa$, then $\nu$ satisfies the desired property.
  
  $(2) \Rightarrow (1)$] We use a standard compactness argument. Suppose that $\Kl$ does not satisfy the metric convex Ramsey property. We can then find structures $\A$, $\B$ in $\Kl$ and $\eps > 0$ such that for every $\C \in \Kl$, there exists a \textit{bad} coloring $\kappa_\C$ of $\emb{A}{C}$, that is $\kappa_\C$ satisfies that for all $\nu \in \left\langle \emb{B}{C} \right\rangle$, the oscillation of $\kappa_\C$ on $\emb{A}{C}(\nu)$ is greater than $\eps$.
  
  We fix an ultrafilter $\Ul$ on the collection of subsets of $\K$ such that for every finite $D \subseteq \K$, the set $\{ E \subseteq \K \text{ finite} : D \subseteq E \}$ belongs to $\Ul$. We consider the map $\displaystyle{\kappa = \lim_\Ul \kappa_\C}$ on $\emb{A}{K}$ defined by
  $$\kappa(\alpha) = t \Leftrightarrow \forall r > 0, \{ \C \subseteq \K \text{ finite}: \kappa_\C(\alpha) \in [t - r, t + r] \} \in \Ul.$$
  Note that the assumption on $\Ul$ implies that for all $\alpha \in \emb{A}{K}$, the set $\{ \C \subseteq \K \text{ finite} : \alpha(A) \subseteq C \}$ is in $\Ul$ so $\kappa_\C(\alpha)$ is defined $\Ul$-everywhere and the above definition makes sense. Besides, since all the $\kappa_\C$ are $1$-Lipschitz, $\kappa$ is too and is thus a coloring of $\emb{A}{K}$. We prove that $\kappa$ contradicts property $(2)$.
  
  Let $\nu \in \left\langle \emb{B}{K} \right\rangle$ and write $\displaystyle{\nu = \sum_{i=1}^n \lambda_i \delta_{\beta_i}}$, with the $\beta_i$'s in $\emb{B}{K}$. Then, for all $i \in \{1,...,n\}$, the sets $\{ \C \subseteq \K \text{ finite} : \beta_i(B) \subseteq \C \}$ belong to $\Ul$ and so does their intersection $U_\nu$. 
  Furthermore, the set $\emb{A}{K}(\nu)$, which is also $\emb{A}{C}(\nu)$ for any $\C$ in $U_\nu$, is finite — note that this isn't true of $\left\langle \emb{A}{K}(\nu) \right\rangle$ (so choosing the definition of remark \ref{mooresecompliquelavie} for the Ramsey property would make the proof technically harder). For every $\C$ in $U_\nu$, there exists $\alpha, \alpha'$ in $\emb{A}{C}(\nu)$ such that $\lvert \kappa_\C(\alpha) - \kappa_\C(\alpha') \rvert \geqslant \eps$. So there exist $\alpha, \alpha'$ in $\emb{A}{K}(\nu)$ such that the set $\{ \C \subseteq \K \text{ finite} : \lvert \kappa_\C(\alpha) - \kappa_\C(\alpha') \rvert \geqslant \eps \}$ belongs to $\Ul$. By definition of $\kappa$, this implies that $\lvert \kappa(\alpha) - \kappa(\alpha') \rvert \geqslant \eps$, which shows that $(2)$ fails for $\nu$. As $\nu$ was arbitrary, this completes the proof.
   \end{proof}

\section{The metric convex Ramsey property for the automorphism group}

Let $G$ be the automorphism group of $\K$.

In this section, we reformulate the metric convex Ramsey property in terms of properties of $G$.

\begin{df}
Let $\A$ be a finite substructure of $\K$. We define a pseudometric $d_\A$ on $G$ by
$$d_\A(g,h) = \sup_{a \in A} d(g(a),h(a)).$$
We will denote by $(G,d_\A)$ the induced metric quotient space.
\end{df}

\begin{rmq}\label{pseudometricsgen}
The pseudometrics $d_\A$, for finite substructures $\A$ of $\K$, generate the left uniformity on $G$.
\end{rmq}
   
   The pseudometric $d_\A$ is the counterpart of the metric $\rho_\A$ on $\emb{A}{K}$ on the group side. More specifically, as pointed out in \cite[lemma 3.8]{julientodor-amenablecontinuouslogic}, the restriction map $\Phi_\A : (G,d_\A) \to (\emb{A}{K}, \rho_\A)$ defined by $g \mapsto g_{\restriction A}$ is distance-preserving and its image $\Phi_\A(G)$ is dense in $\emb{A}{K}$. 
\noindent As a consequence, every $1$-Lipschitz map $f : (G,d_\A) \to \I$ extends uniquely, via $\Phi_\A$, to a coloring $\kappa_f$ of $\emb{A}{K}$, while every coloring $\kappa$ of $\emb{A}{K}$ restricts to a $1$-Lipschitz map $f_\kappa : (G, d_\A) \to \I$.

   \begin{prop}\label{ramseygroupeUC}
   The following are equivalent.
   \ben
   \item The class $\Kl$ has the metric convex Ramsey property.
   \item For every $\eps > 0$, every finite substructure $\A$ of $\K$, every finite subset $F$ of $G$ and every $1$-Lipschitz map $f : (G, d_\A) \to \I$, there exist elements $g_1,..., g_n$ of $G$ and barycentric coefficients $\lambda_1,..., \lambda_n$ such that for all $h,h'$ in $F$, one has
   $$\left\lvert \sum_{i=1}^n \lambda_i f(g_i h) - \sum_{i=1}^n \lambda_i f(g_i h') \right\rvert < \eps.$$
   \item For every $\eps > 0$, every finite subset $F$ of $G$, every left uniformly continuous map $f : G \to \I$, there exist elements $g_1,..., g_n$ of $G$ and barycentric coefficients $\lambda_1,..., \lambda_n$ such that for all $h,h'$ in $F$, one has
   $$\left\lvert \sum_{i=1}^n \lambda_i f(g_i h) - \sum_{i=1}^n \lambda_i f(g_i h') \right\rvert < \eps.$$
   \een
   \end{prop}
   
   \begin{rmq}
   The finite subset $F$ of $G$ in condition $(2)$ is the counterpart of the structure $\B$ in the Ramsey property: by approximate ultrahomogeneity of the limit $\K$, it corresponds, up to a certain error, to the set of all embeddings of $\A$ into $\B$.
   \end{rmq}
   
   \begin{proof}
   $(1) \Rightarrow (2)$] We set $\displaystyle{\B = \A \cup \bigcup_{h \in F} h(\A)}$. Let $\kappa_f$ be the unique coloring of $\emb{A}{K}$ that extends $f$.
   We then apply proposition \ref{ramseyK} to $\A$, $\B$, $\eps$ and $\kappa_f$: there is $\nu$ in $\left\langle \emb{B}{K} \right\rangle$ such that for all $\alpha, \alpha'$ in $\emb{A}{K}(\nu)$, we have $\lvert \kappa_f(\alpha) - \kappa_f(\alpha') \rvert < \eps$.
   
 Write $\displaystyle{\nu = \sum_{i = 1}^n \lambda_i \delta_{\beta_i}}$, with the $\beta_i$'s in  $\emb{B}{K}$. Since the structure $\K$ is a Fraïssé limit, it is approximately ultrahomogeneous. This implies that for each $i$ in $\{1,..., n\}$, there exists an element $g_i$ of its automorphism group $G$ such that $\rho_\B(g_i, \beta_i) < \eps$. We show that the $g_i$'s and $\lambda_i$'s have the desired property.
   
To that end, let $h$ and $h'$ be elements of $F$. Note that, by definition of $\B$, (the restrictions of) $h$ and $h'$ are in $\emb{A}{B}$ so $\nu \circ \delta_{h}$ and $\nu \circ \delta_{h'}$ belong to $\emb{A}{K}(\nu)$. In particular, we have 
$$\left\lvert \sum_{i=1}^n \lambda_i \kappa_f(\beta_i \circ h) - \sum_{i=1}^n \lambda_i \kappa_f(\beta_i \circ h') \right\rvert = \lvert \kappa_f(\nu \circ \delta_h) - \kappa_f(\nu \circ \delta_{h'}) \rvert < \eps.$$
We now compute:
\begin{align*}
\left\lvert \sum_{i=1}^n \lambda_i f(g_i h) - \sum_{i=1}^n \lambda_i f(g_i h') \right\rvert
&\leqslant
\left\lvert \sum_{i=1}^n \lambda_i f(g_i h) - \sum_{i=1}^n \lambda_i \kappa_f(\beta_i \circ h) \right\rvert \\
&+
\left\lvert \sum_{i=1}^n \lambda_i \kappa_f(\beta_i \circ h) - \sum_{i=1}^n \lambda_i \kappa_f(\beta_i \circ h') \right\rvert \\
&+
\left\lvert \sum_{i=1}^n \lambda_i \kappa_f(\beta_i \circ h') - \sum_{i=1}^n \lambda_i f(g_i h') \right\rvert.
\end{align*}
We estimate the first piece as follows.
\begin{align*}
\left\lvert \sum_{i=1}^n \lambda_i f(g_i h) - \sum_{i=1}^n \lambda_i \kappa_f(\beta_i \circ h) \right\rvert
&\leqslant 
\sum_{i=1}^n \lambda_i \lvert f(g_i h) - \kappa_f(\beta_i \circ h) \rvert \\
&= 
\sum_{i=1}^n \lambda_i \lvert \kappa_f(g_i h) - \kappa_f(\beta_i \circ h) \rvert \\
&\leqslant \sum_{i=1}^n \lambda_i \rho_\A(g_i h, \beta_i \circ h)  \text{ because $\kappa_f$ is $1$-Lipschitz}\\
&\leqslant \sum_{i=1}^n \lambda_i \rho_\B(g_i, \beta_i) \text{ by definition of $\rho_\B$}\\
&< \eps.
\end{align*}
Similarly, we obtain 
$$\left\lvert \sum_{i=1}^n \lambda_i \kappa_f(\beta_i \circ h') - \sum_{i=1}^n \lambda_i f(g_i h') \right\rvert < \eps,$$
  hence finally
   $$\left\lvert \sum_{i=1}^n \lambda_i f(g_i h) - \sum_{i=1}^n \lambda_i f(g_i h') \right\rvert < 3\eps.$$

 $(2) \Rightarrow (3)$] We approximate uniformly continuous functions by Lipschitz ones. More precisely, let $f : G \to \I$ be left uniformly continuous and $\eps > 0$. There exists an entourage $V$ in the left uniformity $\Ul_L(G)$ on $G$ such that for all $x,y$ in $G$, if $(x,y) \in V$, then $\lvert f(x) - f(y) \rvert < \eps$. Besides, remark \ref{pseudometricsgen} implies there exist $\A \leqslant \K$ finite and $r > 0$ such that for all $x,y$ in $G$, if $d_\A(x,y) < r$, then $(x,y) \in V$.
   
   Now, for $k \in \N^*$, we can define a mapping $f_k : (G, d_\A) \to \I$ by $\displaystyle{f_k(x) = \inf_{y \in G} f(y) + k d_\A(x,y)}$. It is $k$-Lipschitz as the infimum of $k$-Lipschitz functions. And for a big enough $k$, the fact that $f$ is bounded implies that for every $x$ in $G$, $\lvert f_k(x) - f(x) \rvert < \eps$: we have obtained a good uniform approximation of $f$ by a Lipschitz function. 
      
   We then apply $(2)$ to $\frac{f_k}{k}$, which is $1$-Lipschitz, and to $\frac{\eps}{k}$: for every finite subset $F$ of $G$, there exist elements $g_1,..., g_n$ of $G$ and barycentric coefficients $\lambda_1,...,\lambda_n$ such that for all $h,h' \in F$, we have
   $$\left\lvert \sum_{i=1}^n \lambda_i \frac 1k f_k(g_i h) - \sum_{i=1}^n \lambda_i \frac 1k f_k(g_i h') \right\rvert < \frac{\eps}{k}$$
   hence 
   $$\left\lvert \sum_{i=1}^n \lambda_i f_k(g_i h) - \sum_{i=1}^n \lambda_i f_k(g_i h') \right\rvert < \eps.$$
   Then, for all $h,h' \in F$, the triangle inequality gives
   $$\left\lvert \sum_{i=1}^n \lambda_i f(g_i h) - \sum_{i=1}^n \lambda_i f(g_i h') \right\rvert < 3\eps.$$      
      
   $(3) \Rightarrow (1)$] We use the same compactness argument as in the proof of proposition \ref{ramseyK}: assume that $\Kl$ does not have the metric convex Ramsey property. We can then find $\eps > 0$, structures $\A$ and $\B$ and a coloring $\kappa : \emb{A}{K} \to \I$ such that no measure $\nu$ in $\left\langle \emb{B}{K} \right\rangle$ satisfies $\lvert \kappa(\alpha) - \kappa(\alpha') \rvert < \eps$ for all $\alpha, \alpha'$ in $\emb{A}{K}(\nu)$. 
   
   Now consider the restriction $f_\kappa$ of the coloring $\kappa$ to $(G,d_{\A})$. It is left uniformly continuous from $G$ to $\I$. Since $\Kl$ is approximately ultrahomogeneous, for every $\alpha$ in $\emb{A}{B}$, there is $h_\alpha$ in $G$ such that $\rho_{\A}(h_\alpha, \alpha) < \eps$. Let $F$ be the (finite) set of all such $h_\alpha$'s. We show that $f_\kappa$ fails to satisfy condition $(3)$ for $F$ and $\eps$. 

Indeed, towards a contradiction, assume that there exists elements $g_1,..., g_n$ of $G$ and barycentric coefficients $\lambda_1,..., \lambda_n$ such that for all $h, h'$ in $F$, one has
$$\left\lvert \sum_{i=1}^n \lambda_i f_\kappa(g_i h) - \sum_{i=1}^n \lambda_i f_\kappa(g_i h') \right\rvert < \eps.$$
Set $\displaystyle{\nu = \sum_{i=1}^n \lambda_i \delta_{g_i}\in \left\langle \emb{B}{K} \right\rangle}$. Now pick embeddings $\alpha$ and $\alpha'$ of $\A$ in $\K$; then $\nu \circ \delta_\alpha$ and $\nu \circ \delta_{\alpha'}$ are in $\left\langle \emb{A}{K}(\nu) \right\rangle$. 
Using similar arguments, we have 
$$\left\lvert \kappa(\nu \circ \delta_\alpha) - \sum_{i=1}^n \lambda_i f_\kappa(g_i h_\alpha) \right\rvert < \eps$$
and
$$\left\rvert \sum_{i=1}^n \lambda_i f_\kappa(g_i h_{\alpha'}) - \kappa(\nu \circ \delta_{\alpha'}) \right\rvert < \eps$$
so 
$$\lvert \kappa(\nu \circ \delta_{\alpha}) - \kappa(\nu \circ \delta_{\alpha'}) \rvert < 3\eps.$$

   \end{proof}
   
   Notice that condition $(3)$ does not depend on the Fraïssé class but only on its automorphism group.
   
   By virtue of remark \ref{Ncolorings}, the metric convex Ramsey property is  equivalent to condition $(3)$ for any finite number of colorings at once. It is that condition which will imply amenability in theorem \ref{main}.   
      
Moreover, if $G$ is endowed with a compatible left-invariant metric, Lipschitz functions are uniformly dense in uniformly continuous ones, so we can replace uniformly continuous maps by $1$-Lipschitz maps in condition $(3)$.
      
\section{A criterion for amenability}

Given a compact space $X$, we denote by $P(X)$ the set of all probability measures on $X$. It is a subspace of the dual space of continuous maps on $X$. Indeed, if $\mu$ is in $P(X)$ and $f$ is a continuous function on $X$, we put $\displaystyle{\mu(f) = \int_X f d\mu}$. Moreover, if we endow $P(X)$ with the induced weak* topology, it is compact.

If $G$ is a group that acts on $X$, then one can define an action of $G$ on $P(X)$ by 
$$(g \cdot \mu)(f) = \int_X f(g^{-1} \cdot x) d\mu(x).$$

\begin{df}
A topological group $G$ is said to be \textbf{amenable} if every continuous action of $G$ on a compact space $X$ admits a measure in $P(X)$ which is invariant under the action of $G$.
\end{df}

We are now ready to prove the main theorem. 

\begin{thm}\label{main}
Let $\Kl$ be a metric Fraïssé class, $\K$ its Fraïssé limit and $G$ the automorphism group of $\K$. Then the following are equivalent.
\ben
\item The topological group $G$ is amenable.
\item The class $\Kl$ has the metric convex Ramsey property.
\een
\end{thm}

\begin{proof}
$(1) \Rightarrow (2)$] Suppose $G$ is amenable and let $\A$, $\B$ be structures in the class $\Kl$, $\eps > 0$ and $\kappa_0 : \emb{A}{K} \to \I$ a coloring. We show that there exists $\nu \in \left\langle \emb{B}{K} \right\rangle$ such that for all $\alpha, \alpha' \in \emb{A}{K}(\nu)$, we have $\lvert \kappa_0(\alpha) - \kappa_0(\alpha') \rvert < \eps$, which will imply the metric convex Ramsey property (by proposition \ref{ramseyK}). We adapt Moore's proof to the metric setting.

The group $G$ acts continuously on the compact space $\I^{\emb{A}{K}}$ by $g \cdot \kappa (\alpha) = \kappa(g^{-1} \circ \alpha)$. Denote by $Y$ the orbit of the coloring $\kappa_0$ under this action and by $X$ its closure, which is compact. Note that all the functions in $X$ are colorings as well. We consider the restriction of the action to $X$: since $G$ is amenable, there is an invariant probability measure $\mu$ on $X$.

The invariance of $\mu$ implies that the map $\displaystyle{\alpha \mapsto \int_X \kappa(\alpha) d\mu(\kappa)}$ is constant on $\emb{A}{K}$. To see this, fix $\alpha, \alpha'$ in $\emb{A}{K}$ and $s>0$. Since $\K$ is approximately ultrahomogeneous, we can find $g$ in $G$ such that $\rho_\A(\alpha', g^{-1} \circ \alpha) < s$. Then
\begin{align*}
\left\lvert \int_X \kappa(\alpha') d\mu(\kappa) - \int_X \kappa(\alpha) d\mu(\kappa) \right\rvert
&= \left\lvert \int_X \kappa(\alpha') d\mu(\kappa) - \int_X g \cdot \kappa(\alpha) d\mu(g \cdot \kappa) \right\rvert \\
&= \left\lvert \int_X \kappa(\alpha') d\mu(\kappa) - \int_X \kappa(g^{-1} \circ \alpha) d\mu(\kappa) \right\rvert \\
&\leqslant \int_X \lvert \kappa(\alpha') - \kappa(g^{-1} \circ \alpha) \rvert d\mu(\kappa) \\
&\leqslant \int_X \rho_\A(\alpha', g^{-1} \circ \alpha) d\mu(\kappa) \\
&< s.
\end{align*}
Since $s$ was arbitrary, $\displaystyle{\int_X \kappa(\alpha) d\mu(\kappa) = \int_X \kappa(\alpha') d\mu(\kappa)}$. Let $r$ denote this constant value.

Besides, $Y$ being dense in $X$, the collection of finitely supported probability measures on $Y$ is dense in $P(X)$. In particular, there exist barycentric coefficients $\lambda_1,..., \lambda_n$ and elements $g_1,...,g_n$ of $G$ such that for all $\alpha$ in $\emb{A}{K}$, we have $\displaystyle{\left\lvert \sum_{i=1}^n \lambda_i \kappa_0(g_i^{-1} \circ \alpha) - r \right\rvert < \eps}$.

Finally, we may assume that $\B$ is a substructure of $\K$, and set $\beta_i = g_i^{-1} \restriction \B$, for $i$ in $\{1,...,n\}$, and $\nu = \sum_{i=1}^n \lambda_i \delta_{\beta_i} \in \left\langle \emb{B}{K} \right\rangle$. Then $\nu$ as is desired. Indeed, if $\alpha, \alpha'$ are in $\emb{A}{B}$, and thus in $\emb{A}{K}$, we have
\begin{align*}
\lvert \kappa_0(\nu \circ \delta_\alpha) - \kappa_0(\nu \circ \delta_\alpha') \rvert
&= \left\lvert \sum_{i=1}^n \lambda_i \kappa_0(\beta_i \circ \alpha) - \sum_{i=1}^n \lambda_i \kappa_0(\beta_i \circ \alpha') \right\rvert \\
&\leqslant \left\lvert \sum_{i=1}^n \lambda_i \kappa_0(\beta_i \circ \alpha) - r \right\rvert 
+ \left\lvert r - \sum_{i=1}^n \lambda_i \kappa_0(\beta_i \circ \alpha') \right\rvert \\
&=  \left\lvert \sum_{i=1}^n \lambda_i \kappa_0(g_i^{-1} \circ \alpha) - r \right\rvert 
+ \left\lvert r - \sum_{i=1}^n \lambda_i \kappa_0(g_i^{-1} \circ \alpha') \right\rvert \\
&< 2\eps.
\end{align*}

$(2) \Rightarrow (1)$] Conversely, suppose that $\Kl$ has the metric convex Ramsey property and let $G$ act continuously on a compact space $X$. We show that $X$ admits an invariant probability measure. Since $P(X)$ is compact, it suffices to show that if $f_1,..., f_N : X \to \I$ are uniformly continuous, $\eps > 0$ and $F$ is a finite subset of $G$, there exists $\mu$ in $M(X)$ such that for all $j$ in $\{1,...,N\}$ and all $h$ in $F$, $\lvert h \cdot \mu(f_j) - \mu(f_j) \rvert < \eps$. 

Fix $x$ in $X$. For $j$ in $\{1,..., N\}$, we lift $f_j$ to a map $\tilde f_j : G \to \I$ by setting $\tilde f_j(g) = f_j(g^{-1} \cdot x)$. Since the action of $G$ on $X$ is continuous and $f_j$ is uniformly continuous, the map $\tilde f_j$ is left uniformly continuous. 

We then apply proposition \ref{ramseygroupeUC} to $F \cup \{1\}$, $\eps$ and $\tilde f_1,..., \tilde f_N$ to obtain barycentric coefficients $\lambda_1,..., \lambda_n$ and elements $g_1,...,g_n$ of $G$ such that for all $j$ in $\{1,...,N\}$, for all $h$ in $F$ (and $h' = 1$), we have
$$\left\lvert \sum_{i=1}^n \lambda_i \tilde f_j(g_i h) - \sum_{i=1}^n \lambda_i \tilde f_j(g_i) \right\rvert < \eps.$$
Then $\mu = \sum_{i=1}^n \lambda_i \delta_{g_i^{-1} \cdot x}$ is as desired. Indeed, let $j \in \{1,...,N\}$ and $h \in F$. We have
$$\mu(f_j) = \sum_{i=1}^n \lambda_i f_j(g_i^{-1} \cdot x) = \sum_{i=1}^n \lambda_i \tilde f_j(g_i)$$
and
\begin{align*}
h \cdot \mu(f_j) 
&= \sum_{i=1}^n \lambda_i (h \cdot f_j)(g_i^{-1} \cdot x) \\
&= \sum_{i=1}^n \lambda_i f_j(h^{-1} g_i^{-1} \cdot x) \\
&= \sum_{i=1}^n \lambda_i \tilde f_j(g_i h)
\end{align*}
so finally
$$\left\lvert h \cdot \mu(f_j) - \mu(f_j) \right\rvert = \left\lvert \sum_{i=1}^n \lambda_i \tilde f_j(g_i h) - \sum_{i=1}^n \lambda_i \tilde f_j(g_i) \right\rvert < \eps,$$ which completes the proof.
\end{proof}

As a consequence of the remark following proposition \ref{ramseygroupeUC} and of the fact that every Polish group is the automorphism group of some metric Fraïssé structure (\cite[theorem 6]{MR2767973}), we obtain the following intrinsic characterization of amenability.

\begin{thm}\label{equiv}
Let $G$ be a Polish group and $d$ a left-invariant metric on $G$ which induces the topology. Then the following are equivalent.
\ben
\item The topological group $G$ is amenable.
\item For every $\eps > 0$, every finite subset $F$ of $G$, every $1$-Lipschitz map $f : (G,d) \to \I$, there exist elements $g_1,..., g_n$ of $G$ and barycentric coefficients $\lambda_1,..., \lambda_n$ such that for all $h,h' \in F$, one has
   $$\left\lvert \sum_{i=1}^n \lambda_i f(g_i h) - \sum_{i=1}^n \lambda_i f(g_i h') \right\rvert < \eps.$$
\een
\end{thm}

It implies that amenability is a $G_\delta$ condition in the following sense (see \cite[theorem 3.1]{julien-todor-genericrepresentations}).

\begin{cor}
Let $\Gamma$ be a countable group and $G$ a Polish group. Then the set of representations of $\Gamma$ in $G$ whose image is an amenable subgroup of $G$ is $G_\delta$ in $\text{Hom}(\Gamma,G)$.
\end{cor}

\begin{proof}
Let $\pi$ be a representation of $\Gamma$ into $G$ and $d$ be a left-invariant metric on $G$ which induces the topology. Then, by virtue of theorem \ref{equiv}, $\pi(\Gamma)$ is amenable if and only if for every $\eps > 0$, every finite subset $F$ of $\pi(\Gamma)$, every $1$-Lipschitz function $f : (\pi(\Gamma),d) \to \I$, there exist elements $g_1,...,g_n$ of $\pi(\Gamma)$ and barycentric coefficients $\lambda_1,..., \lambda_n$ such that for all $h,h'$ in $F$, one has
$$\left\lvert \sum_{i=1}^n \lambda_i f(g_i h) - \sum_{i=1}^n \lambda_i f(g_i h') \right\rvert < \eps.$$
Using the same compactness argument as in proposition \ref{ramseyK}, one can show that the condition is equivalent to the following.

\begin{align*}
&\forall \eps > 0, \forall F \subseteq \pi(\Gamma) \text{ finite }, \exists K \subseteq \pi(\Gamma) \text{ finite },\forall f: (KF,d) \to \I \text{ $1$-Lipschitz }, \\
&\exists k_1, ..., k_n \in K, \exists \lambda_1,..., \lambda_n, \forall h,h' \in F,
\left\lvert \sum_{i=1}^n \lambda_i f(k_i h) - \sum_{i=1}^n \lambda_i f(k_i h') \right\rvert < \eps.
\end{align*}

It is easily seen that this is again equivalent to the following.
\[
\forall \eps > 0, \forall F \subseteq \Gamma \text{ finite }, \exists K \subseteq \Gamma \text{ finite },\]
\[
(*)\left\{\begin{array}{l} \forall f: KF \to \I \text{ such that } \forall \gamma, \gamma' \in KF, 
\lvert f(\gamma) - f(\gamma') \rvert \leqslant d(\pi(\gamma), \pi(\gamma')),\\  \exists k_1, ..., k_n \in K, \exists \lambda_1,..., \lambda_n, \forall h,h' \in F,
\displaystyle{\left\lvert \sum_{i=1}^n \lambda_i f(k_i h) - \sum_{i=1}^n \lambda_i f(k_i h') \right\rvert < \eps}.\end{array}\right.
\]

We now prove that, if $\eps$, $F$ and $K$ are fixed, the set of representations $\pi$ satisfying condition $(*)$ above is open, which will imply that the condition is indeed $G_\delta$.
We prove that its complement is closed. To that aim, take a sequence $(\pi_k)$ of representations in the complement that converges to some representation $\pi$. Let $f_k : KF \to \I$ witness that $\pi_k$ is in the complement. Since $KF$ is finite, maps from $KF$ to $\I$ form a compact set so we may assume that $(f_k)$ converges to some $f$. Since the Lipschitz condition is closed, $f$ also satifies that for all $\gamma, \gamma'$ in $KF$, $\lvert f(\gamma) - f(\gamma') \rvert \leqslant d(\pi(\gamma), \pi(\gamma'))$.

By the choice of $f_k$, for all $k_1,..., k_n$ in $K$ and all $\lambda_1,..., \lambda_n$, there exists $h_k, h'_k$ in $F$ such that
$$\left\lvert \sum_{i=1}^n \lambda_i f_k(k_i h_k) - \sum_{i=1}^n \lambda_i f_k(k_i h'_k) \right\rvert \geqslant \eps.$$
Since $F$ is finite, we may again assume that there are $h$ and $h'$ in $F$ such that for all $k$, we have $h_k = h$ and $h'_k = h'$. We then take the limit of the above inequality to get that
$$\left\lvert \sum_{i=1}^n \lambda_i f(k_i h) - \sum_{i=1}^n \lambda_i f(k_i h') \right\rvert \geqslant \eps,$$
which implies that $\pi$ does not satisfy condition $(*)$ either and thus completes the proof.
\end{proof}

This yields the following criterion for amenability, which can however be obtained without the use of Ramsey theory.

\begin{cor}\label{concretecriterion}
Let $G$ be a Polish group such that for every $n$ in $\N^*$, the set
$$F_n = \{ (g_1,..., g_n) \in G^n : \langle g_1,...,g_n \rangle \text{ is amenable} \}$$
is dense in $G^n$. Then $G$ is amenable.
\end{cor}

\begin{proof}
We use a Baire category argument.
By virtue of the above corollary, for all $n$, the set $F_n$ is dense $G_\delta$ in $G^n$. As a result, the set
$$F = \{ (g_k) \in G^\N : \forall n, (g_1,..., g_n) \in F_n \}$$
is dense and $G_\delta$ too.
Besides, the set of sequences which are dense in $G$ is also dense and $G_\delta$. Then the Baire category theorem gives a sequence $(g_k)$ in their intersection. Thus, the group generated by the $g_k$'s is dense and amenable and so is $G$.
\end{proof}

Note that since compact groups are amenable, it implies in particular that a group in which the tuples that generate a compact subgroup are dense is amenable. 

\begin{rmq*}
The criterion of corollary \ref{concretecriterion} can also be proven directly using the following compactness argument. Let $G$ act continuously on a compact space $X$. For every finite subset $F$ of $G$ and every entourage $V$ in the uniformity on $P(X)$, we approximate the elements of $F$ by a tuple in some $F_n$ to find a measure $\mu_{F,V}$ in $P(X)$ which is $V$-invariant by every element of $F$. Since $P(X)$ is compact, the net $\{ \mu_{F,V} \}$ admits a limit point, which is invariant under the action of $G$.

Note that, in view of this direct argument, it is enough to ask that the set of tuples which generate a subgroup that is \textit{included} in an amenable one be dense.

The same argument works with extreme amenability as well and it allows to slightly simplify the arguments of \cite{julien-todor-genericrepresentations}: to show that the groups $\text{Iso}(\U)$, $U(H)$ and $\text{Aut}(X,\mu)$, Melleray and Tsankov use their theorem 7.1 along with the facts that extreme amenability is a $G_\delta$ property and that Polish groups are generically $\aleph_0$-generated. This is not necessary, as the core of their proof is basically the above criterion: in each case, they prove that the set of tuples which generate a subgroup that is contained in an extremely amenable group (some $L^0(U(m))$, as it happens) is dense.
\end{rmq*}

\section{Concluding remarks}

One would expect the characterization of theorem \ref{main} to yield new examples of amenable groups or at least simpler proofs of the amenability of known groups. However, proving the convex Ramsey property for a concrete Fraïssé class is quite technical and difficult. Indeed, we have no example of a metric class which satisfies the metric convex Ramsey property but not the metric (approximate) Ramsey property.

The only such example that we know of is discrete: it is the class of finite sets with no additional structure. It is well known that its automorphism group, $S_\infty$, is amenable but not extremely amenable. In fact, the class of finite sets has the classical Ramsey property (it follows from the Ramsey theorem), but since finite sets are not rigid (every permutation is an automorphism), Kechris, Pestov and Todor\v{c}evi\'{c}'s result does not apply. However, we can still use this classical Ramsey property to recover the amenability: we circumvent the problem of non-rigidity by averaging the colors of all permutations of the small structure $A$ to obtain the convex Ramsey property. We do not know if this technique generalizes to other nonrigid classes.

Maybe our characterization can be used the other way round, that is, to find new Ramsey-type results. There is also hope that the criterion of corollary \ref{concretecriterion} may lead to (new) examples of amenable groups. 

\subsection*{Acknowledgements}

My heartfelt thanks go to Julien Melleray for his constant help and guidance as well as for various improvements on this paper. I am also very grateful to Lionel Nguyen Van Thé with whom I have had insightful discussions about Ramsey theory. Finally, I would like to thank François Le Maître for pointing out a direct proof of corollary 19.

\bibliographystyle{plain}
\bibliography{/Users/adrianekaichouh/Dropbox/Maths/Bibliographie/biblio}

\end{document}